\documentclass{article}

\usepackage[english]{babel}
\usepackage{amsmath, amsxtra, latexsym,amscd, amsthm, amssymb,indentfirst}
\usepackage[mathscr]{eucal}
\usepackage{enumerate}
\usepackage{textcomp}
\usepackage{array, tabularx, longtable}%
\usepackage{multicol}
\usepackage{color}

\usepackage[margin=1.4in]{geometry} 

\usepackage{float}

\usepackage{caption}
\usepackage{graphicx}

\def\R{\mathbb R}
\def\N{\mathbb N}

\newcommand*{\prox}{\rm{prox}}

\def\int{{\rm int\,}}

\theoremstyle{plain}
\newtheorem{theorem}{Theorem}[section]
\theoremstyle{theorem}
\numberwithin{theorem}{section}

\newtheorem{lemma}[theorem]{Lemma}
\newtheorem{proposition}[theorem]{Proposition}
\theoremstyle{definition}
\newtheorem{definition}[theorem]{Definition}

\newtheorem{example}[theorem]{Example}
\newtheorem{remark}[theorem]{Remark}

\title{A novel  neural network   with predefined-time stability for solving generalized monotone inclusion problems with applications 
}
\author{ Nam V. Tran\thanks{ Faculty of Applied Sciences, University of Technology and Engineering HCMC, Ho Chi Minh City, Vietnam; email: namtv@hcmute.edu.vn}
}

\begin{document}
	\maketitle
	

	
	\abstract{We propose a forward--backward splitting dynamical framework for solving inclusion
		problems of the form \(0 \in F(x) + G(x)\) in Hilbert spaces, where \(F\) is a
		maximal set-valued operator and \(G\) is a single-valued mapping. The analysis is
		conducted under a generalized monotonicity assumption, which relaxes the
		classical monotonicity conditions commonly imposed in the literature and thereby
		extends the applicability of the proposed approach.
		
		Under mild conditions on the system parameters, we establish both fixed-time and
		predefined-time stability of the resulting dynamical system. The fixed-time
		stability guarantees a uniform upper bound on the settling time that is
		independent of the initial condition, whereas the predefined-time stability
		framework allows the system parameters to be selected \emph{a priori} in order
		to ensure convergence within a user-specified time horizon.
		
		Moreover, we investigate an explicit forward Euler discretization of the
		continuous-time dynamics, leading to a novel forward--backward iterative
		algorithm. A rigorous convergence analysis of the resulting discrete scheme is
		provided. Finally, the effectiveness and versatility of the proposed method are
		illustrated through several classes of problems, including constrained
		optimization problems, mixed variational inequalities, and variational
		inequalities, together with numerical experiments that corroborate the
		theoretical results.

	}
	

	
	
	\maketitle
	
	\section{Introduction}
Let $ \mathscr{H}$ be a real finite-dimensional Hilbert space endowed with the inner product 
$\langle \cdot,\cdot \rangle$ and the induced norm $\|\cdot\|$. 
This paper is devoted to the study of the following inclusion problem: find $z^*\in \mathscr{H}$ such that
\begin{equation}\label{inc1}
	0 \in F(z^*) + G(z^*),
\end{equation}
where $F~:\mathscr{H}\to 2^\mathscr{H}$ is a set-valued operator and $G:\mathscr{H}\to \mathscr{H}$ is a single-valued operator. 
Problem \eqref{inc1} provides a unified modeling framework for numerous mathematical problems arising in optimization, variational inequalities, equilibrium and saddle-point problems, Nash equilibrium problems in noncooperative games, and fixed-point problems; see, for example, \cite{BlumOettli94, Cav, NHVF}.

A powerful approach for solving inclusion problems relies on continuous-time dynamical systems. 
Due to their simple structure and low computational burden, such systems have been extensively studied and successfully applied to inclusion problems, mixed variational inequalities, variational inequalities, and constrained optimization problems; see, e.g., 
\cite{AAS, BC_SICON, BCJMAA, BotCV, JU4, LIU, VS20, ZHU}. 
Representative contributions include the study of global exponential stability for forward--backward--forward dynamical system \cite{22}, the analysis of second-order forward--backward dynamics for inclusion problems \cite{BC_SICON}, and the investigation of Newton-like inertial dynamics with Tikhonov regularization \cite{BOT3}. 
Nevertheless, the majority of existing works focus on asymptotic stability or exponential stability of the associated dynamical systems \cite{27,26,29,18,20,21,22}.

Motivated by the notion of finite-time convergence introduced in \cite{BHAT}, recent studies have proposed dynamical systems that converge to an equilibrium within a finite settling time \cite{JU6, LI}. 
However, in these systems the settling time typically depends on the initial condition and may become arbitrarily large as the initial state moves away from the solution. 
This dependence significantly limits their applicability in practice, especially in scenarios such as robotics and networked systems where accurate initial states are often unavailable. 
To address this issue, the concept of fixed-time convergence was introduced, guaranteeing convergence within a uniform upper bound independent of the initial condition. 
This notion was first proposed by Polyakov \cite{polyakov} and has since been further developed in a variety of contexts; see, for instance, \cite{Garg21, JU3, SAN, WANG, WANG2, WANG3, ZUO}.

Finite-time and fixed-time stability have found wide applications in optimization and control, including multi-agent systems \cite{ZUO}, sparse optimization \cite{GARG4, HE, YU}, unconstrained and constrained optimization \cite{CHEN, GAR, ROMEO}, and equilibrium problems \cite{JU3, JU5, JU6}. 
Within the framework of inclusion problems, finite-time and fixed-time stability properties have been investigated using first-order and proximal-type dynamical systems \cite{NH, NHVF}. 
In particular, fixed-time stability results for generalized monotone inclusion problems were established in \cite{NHVF}.

More recently, an even stronger notion of stability, referred to as \emph{predefined-time stability}, was introduced in \cite{Sanchez2014}. 
In contrast to fixed-time stability, predefined-time stability allows the settling time to be explicitly prescribed in advance via suitable system parameters, while remaining independent of the initial condition. 
This property provides an important advantage in applications requiring strict convergence-time guarantees.

Despite the growing literature on finite-time and fixed-time stability, predefined-time stability for generalized monotone inclusion problems has not yet been studied. 
This paper aims to bridge this gap.

Specifically, we investigate a forward--backward splitting first-order dynamical system derived from a fixed-point reformulation of \eqref{inc1} (see, e.g., \cite{AM, Bot, BSV, V0}). 
Under generalized monotonicity assumptions, we establish both fixed-time and predefined-time stability of the proposed dynamical system. 
In addition, we propose a discrete-time counterpart leading to a relaxed forward--backward algorithm and prove the linear convergence of the generated sequence to the unique solution of the inclusion problem.

It is worth emphasizing that classical monotonicity assumptions are often restrictive, as in many applications the associated operators fail to satisfy this property. 
In this work, we adopt a generalized monotonicity framework that allows the monotonicity modulus to take negative values, thereby extending the applicability of the proposed approach to a broader class of operators.

The main contributions of this paper can be summarized as follows.
\begin{enumerate}
	\item We propose a forward--backward splitting dynamical system exhibiting fixed-time and predefined-time stability for solving inclusion problems, together with explicit upper bounds on the settling time.
	
	\item We demonstrate the applicability of the proposed framework to constrained optimization problems, mixed variational inequalities, and variational inequalities by treating them as special cases of inclusion problems.
	
	\item Our analysis is conducted under generalized monotonicity assumptions, relaxing the classical monotonicity and strong monotonicity requirements commonly imposed in the literature.
\end{enumerate}
The remaining sections of the article are organized as follows.
Section~\ref{Preliminaries} introduces basic definitions and preliminary results.
Section~\ref{sec3} is devoted to the presentation and analysis of the proposed forward--backward dynamical system.
Section~\ref{sec4} studies consistent discretizations and convergence properties of the discrete scheme.
Section~\ref{sec5} discusses applications to constrained optimization problems, mixed variational inequalities, and variational inequalities.
Numerical experiments are presented  in Section~\ref{num}.
Finally, concluding remarks are given in Section~\ref{sec6}.
\section{Preliminaries} \label{Preliminaries}

This section recalls several basic notions and preliminary results from convex analysis and operator theory that will be used throughout the paper.

\subsection{Preliminary Notions on functions}
A proper, convex, and lower semicontinuous (l.s.c.) function $h: \mathscr{H}\to\mathbb{R}\cup\{+\infty\}$ is said to be \emph{subdifferentiable} at a point $z\in \mathscr{H}$ if the set
\[
\partial h(z)
=
\big\{
u\in \mathscr{H}:\;
h(w)\ge h(z)+\langle u,w-z\rangle,
\ \forall y\in \mathscr{H}
\big\}
\]
is nonempty. We call the set $\partial h(z)$ \emph{subdifferential} of $ \mathscr{H}$ at $z$, $u\in\partial h(z)$ is referred to as a subgradient of $ \mathscr{H}$ at $z$. 
We say that the function $ \mathscr{H}$ is subdifferentiable on $ \mathscr{H}$ if it is subdifferentiable at every point of $ \mathscr{H}$.

Observe that if $ \mathscr{H}$ is convex, l.s.c., and has full domain, then it is continuous on $ \mathscr{H}$ \cite[Corollary~8.30]{BauschkeCombettes}. 
Consequently, $ \mathscr{H}$ is subdifferentiable everywhere on $ \mathscr{H}$ \cite[Proposition~16.14]{BauschkeCombettes}. 
Moreover, if $\phi$ and $ \mathscr{H}$ are proper, convex, and l.s.c. functions satisfying 
$\mathrm{dom}\phi\cap\mathrm{int\,dom}h\neq\emptyset$ or $\mathrm{dom}(h)=H$, where
$\mathrm{dom}\phi=\{z\in \mathscr{H}:\,\phi(z)<+\infty\}$, then the subdifferential sum rule holds:
\[
\partial(\phi+h)=\partial \phi+\partial h
\]
(see \cite[Corollary~16.38]{BauschkeCombettes}).

Let $\Omega\subset H$ be a nonempty, closed, and convex set. The normal cone to $C$ at a point $z\in \Omega$ is defined by
\[
N_{\Omega}(z)
=
\big\{
w\in \mathscr{H}:\;
\langle w,z-w\rangle\ge 0,
\ \forall w\in \Omega 
\big\},
\]
and $N_{\Omega}(z)=\emptyset$ whenever $z\notin \Omega$. The indicator function of $\Omega$ is given by
$i_{\Omega}(z)=0$ if $z\in \Omega$ and $i_{\Omega}(z)=+\infty$ otherwise. It is well known that
$\partial i_{\Omega}(z)=N_{\Omega}(z)$ for all $z\in \mathscr{H}$.

For any $z\in \mathscr{H}$, the metric projection of $z$ onto $\Omega$ is defined as
\[
P_{\Omega}(z)
=
\arg\min_{w\in \Omega}\|w-z\|.
\]
It is well known that the projection $P_{\Omega}(z)$ is well defined and unique whenever $\Omega$ is a closed and convex set.

Finally, let $\phi: \mathscr{H}\to\overline{\mathbb{R}}$. The proximity operator of $\phi$ with parameter $\gamma>0$ is defined by
\[
\mathrm{prox}_{\gamma \phi}(z)
=
\arg\min_{w\in \mathscr{H}}
\left\{
\phi(w)+\frac{1}{2\gamma}\|z-w\|^2
\right\},
\qquad z\in \mathscr{H}.
\]
Further background on these concepts can be found in \cite{BauschkeCombettes}.

\subsection{Preliminary Notions on Operators}

Let $F~:\mathscr{H}\to 2^\mathscr{H}$ be a set-valued operator. Its domain is defined by
\[
\mathrm{dom}\,F
=
\{z\in \mathscr{H}:\,F(z)\neq\emptyset\},
\]
and its graph is given by
\[
Gr(F)=\{(z,u)\in \mathscr{H}\times \mathscr{H}:\,u\in F(z)\}.
\]

\begin{definition} We call an operator $F~:\mathscr{H}\to 2^\mathscr{H}$  \emph{$\eta_F$-monotone} if there exists a constant
	$\eta_F\in\mathbb{R}$ such that
	\[
	\langle z-y,u-w\rangle
	\ge
	\eta_F\|z-y\|^2,
	\quad
	\forall z,y\in \mathscr{H},\ u\in F(z),\ w\in F(y).
	\]
\end{definition}

\begin{definition}
A $\eta_F$-monotone operator $F$ is called \emph{maximal} if its graph is not properly contained in the graph of any other $\eta_F$-monotone operator $F': \mathscr{H}\to 2^\mathscr{H}$ \cite{Minh}.
\end{definition}

\begin{remark}
	The constant $\eta_F$ may take negative values. The case $\eta_F=0$ corresponds to classical monotonicity, while $\eta_F>0$ yields strong monotonicity. If $\eta_F<0$, the operator is referred to as weakly monotone.
\end{remark}

A fundamental example of a maximal monotone operator is the subdifferential $\partial f$ of a proper, convex, and l.s.c. function $F~:\mathscr{H}\to(-\infty,+\infty]$, $f\not\equiv+\infty$; see, for example, \cite{ROCK}.

\begin{definition}
	An operator $G: \mathscr{H}\to \mathscr{H}$ is said to be
	\begin{itemize}
		\item Lipschitz continuous with constant $L\ge 0$ if
		$\|G(z)-G(w)\|\le L\|z-w\|$ for all $z,w\in \mathscr{H}$;
		\item $\lambda$-cocoercive if
		\[
		\langle G(z)-G(w),z-w\rangle
		\ge
		\lambda\|G(z)-G(z)\|^2
		\]
		for all $z,w\in \mathscr{H}$.
	\end{itemize}
\end{definition}

By the Cauchy--Schwarz inequality, any $\lambda$-cocoercive operator is $\frac{1}{\lambda}$-Lipschitz continuous.

The resolvent of an operator $F~:\mathscr{H}\to 2^\mathscr{H}$ with parameter $\gamma>0$ is defined as
\[
J_{\gamma F}
=
(Id+\gamma F)^{-1},
\]
where $Id$ denotes the identity operator. Without monotonicity, the resolvent of a maximal operator may fail to be single-valued.

\begin{lemma}\label{GENMON}\cite{Minh}
	Consider an $\eta_F$-monotone set-valued mapping $F : \mathscr{H} \to 2^\mathscr{H}$ and a single-valued operator $G : \mathscr{H} \to \mathscr{H} $, along with a positive parameter $\gamma$ such that $1 + \gamma \eta_F > 0$. Then the following hold:
	\begin{enumerate}
		\item $J_{\gamma F}$ is single-valued;
		\item $J_{\gamma F}$ is $(1+\gamma\eta_F)$-cocoercive;
		\item $\mathrm{dom}\,J_{\gamma F}=H$ if and only if $F$ is maximal $\eta_F$-monotone.
	\end{enumerate}
\end{lemma}

We denote by $\mathrm{Fix}(F)$ the set of fixed points of an operator $F$, and by $\mathrm{zer}\,F$ its set of zeros.

\begin{lemma}\label{lem fixed point}\cite{NHVF}
	Consider an $\eta_F$-monotone set-valued mapping $F : \mathscr{H} \to 2^\mathscr{H}$ and a single-valued operator $G : \mathscr{H} \to \mathscr{H} $, along with a positive parameter $\gamma$ such that $1 + \gamma \eta_F > 0$. Then $z^*\in \mathscr{H}$ solves the inclusion problem \eqref{inc1} if and only if
	\[
	J_{\gamma F}(z^*-\gamma G(z^*))=z^*,
	\]
	that is,
	$z^*\in\mbox{\rm Fix}\left(J_{\gamma F}(Id-\gamma G)\right)$.
\end{lemma}

For additional results on monotone operators and their applications in 
optimization and variational analysis, see 
\cite{AVR,BauschkeCombettes,Minh}.

We denote by $\mathrm{Fix}(F)$ the set of fixed points of an operator $F$, and by 
$\mathrm{zer}\,F$ its set of zeros.

\begin{lemma}\label{lem fixed point}\cite{NHVF}
	Consider an $\eta_F$-monotone set-valued mapping $F : \mathscr{H} \to 2^\mathscr{H}$ and a single-valued operator $G : \mathscr{H} \to \mathscr{H} $, along with a positive parameter $\gamma$ such that $1 + \gamma \eta_F > 0$. Then $z^*\in \mathscr{H}$ solves the inclusion 
	problem \eqref{inc1} if and only if
	\[
	J_{\gamma F}(z^*-\gamma G(z^*))=z^*,
	\]
	that is, $z^*\in\mbox{\rm Fix}\left(J_{\gamma F}(Id-\gamma G)\right)$.
\end{lemma}

In the sequel we will use the below assumption.

\begin{itemize}
	\item [{\bf(A)}]
	$F$ is maximal $\eta_F$-monotone, $G$ is $\eta_G$-monotone and $L$-Lipschitz continuous and the parameter $\gamma>0$ such that   
	$1+\gamma \eta_F>0 $ and $ 2(\eta_F + \eta_G) + \gamma \eta_F^2 - \gamma L^2 > 0. 
	$	
\end{itemize}
\begin{remark}
	As shown in \cite{NHVF}, this assumption is satisfied if one of the following holds: 
	\begin{enumerate}[(a)]
		\item   $\eta_F+\eta_G>0. $
		\item $\eta_F+\eta_G=0$ and  $\eta_F^2>L^2$.
		\item  $\eta_F+\eta_G<0 $ and $\eta_F>L$.
	\end{enumerate}
\end{remark}
It is worth noting that the conditions stated in {\bf (A)} do not necessarily guarantee the (strong) monotonicity of the operator $F+G$. 
In particular, when $F$ is maximal monotone (i.e., $\eta_F = 0$) and $G$ is strongly monotone (i.e., $\eta_G > 0$), 
assumption {\bf (A)} is satisfied for any $\gamma \in (0, 2\eta_G/L^2)$.

The following proposition plays an important role in our convergence analysis of the proposed method.  

\begin{lemma} \cite{NHVF}\label{tr1}  Consider an $\eta_F$-monotone set-valued mapping $F : \mathscr{H} \to 2^\mathscr{H}$ and a single-valued operator $G : \mathscr{H} \to \mathscr{H} $, along with a positive parameter $\gamma$ such that  Assumption {\bf (A)} holds.  	Then,  there exists $c\in (0,1)$  such that 
	\begin{enumerate}[(a)]
		\item \label{patr1} The operator $B:=J_{\gamma F}\circ (Id-\gamma G)$ is Lipchitz continuous with constant $c$.
		\item\label{pbtr1} Fixed points of the operator $B$ exist and are uniquely determined. 
		\item\label{pctr1} $\|B(z)- z^*\|\leq c\|z-z^*\|$ for all $z^*\in \mbox{\rm Fix}(B)$. 
		\item \label{pdtr1} $	(1-c)\|z-z^*\|^2\leq \langle z-z^*, \Phi(z) \rangle $ for all \(z\in \mathscr{H}\setminus \{z^*\}\). 
		\item \label{petr1} $(1-c)\|z-z^*\|\leq \|\Phi(z)\|\leq (1+c)\|z-z^*\|.$
		\item \label{pftr1} $\langle z-z^*, \Phi(z) \rangle \leq (1+c)\|z-z^*\|^2 $ for all \(z\in \mathscr{H}\setminus \{z^*\}\)
	\end{enumerate}
\end{lemma}

\begin{proof}
	Parts \eqref{patr1}-\eqref{pctr1} were proved in \cite{NHVF}. We now prove Part \eqref{pdtr1}. 
	For all \(z\in \mathscr{H}\setminus \{z^*\},\) it holds that 
	\begin{equation}\label{eqn1} 
		\langle z-z^*, x-J_{\gamma F}(z-\gamma G(z)) \rangle =\|z-z^*\|^2+\langle z-z^*, z^*-J_{\gamma F} (z-\gamma G(z))\rangle.
	\end{equation}
	The Cauchy-Schwarz inequality and Part \eqref{pctr1}  yield
	\begin{equation}\label{eqn2} 
		\langle z-z^*, z^*-J_{\gamma F} (z-\gamma G(z))\rangle = \langle z-z^*, z^*-B(z)\rangle\geq -\|x-z^*\|\|z^*-B(z)\| \geq -c \|x-z^*\|^2 
	\end{equation} for some \(c\in (0,1)\). 
	Combining \eqref{eqn1} and \eqref{eqn2} we derive that
	$$\langle z-z^*, \Phi(z) \rangle \geq (1-c)\|z-z^*\|$$
	Hence, Part \eqref{pdtr1} is proved. 
	
	\noindent  As for  Part \eqref{petr1} one has  
	\begin{align*}
		\|\Phi(z)\| =\|\Phi(z)-\Phi(z^*)\|=\|z-z^*-\left(B(z)-B(z^*)\right)\| \leq \|z-z^*\|+\|B(z)-B(z^*)\| 
	\end{align*}
	Because $B$ is Lipschitz continuous with constant $c$, it follows that 
	\begin{align*}
		\|\Phi(z)\|\leq (1+c)\|z-z^*\|.
	\end{align*}
	Also, from  Part \eqref{pdtr1}, we infer that 
	\begin{align*} (1-c)\|z-z^*\|^2\leq \langle z-z^*, \Phi(z) \rangle \leq \|z-z^*\|\|\Phi(z)\| 
	\end{align*} which implies that 
	$$(1-c)\|z-z^*\|\leq \|\Phi(z)\|\quad  \forall z\neq z^*.$$
	Thus, Part \eqref{petr1} is proved.   Part \eqref{pftr1} follows from Part \eqref{petr1}. Hence Proposition is completely proved.  
\end{proof}

The next result presents conditions for the existence and uniqueness of zero points of the sum of two generalized monotone operators. 


\begin{theorem}\label{rem2} 
	Consider an $\eta_F$-monotone set-valued mapping $F : \mathscr{H} \to 2^\mathscr{H}$ and a single-valued operator $G : \mathscr{H} \to \mathscr{H} $, along with a positive parameter $\gamma$ such that $1 + \gamma \eta_F > 0$. Assume further that Assumption {\bf (A)}  holds.     Then the  inclusion problem \eqref{inc1} admits a unique solution.  
\end{theorem} 
\begin{proof}
	By Part \eqref{pbtr1} of Lemma \ref{tr1}, the operator $B=J_{\gamma F}\circ (Id-\gamma G)$ has a unique fixed point.  It follows from Lemma \ref{lem fixed point} that the inclusion problem \eqref{inc1} has a unique solution. 
\end{proof}
\begin{remark}
	If parameter $\eta_F, \eta_G, L, \gamma$ satisfy one of the following: 
	\begin{enumerate}[(a)]
		\item $\eta_F+\eta_G>0, $
		\item  $\eta_F+\eta_G=0$ and  $\eta_F^2>L^2$,
		\item $\eta_F+\eta_G<0 $ and $\eta_F>L$.
	\end{enumerate}   
	Then Assumption {\bf (A)} holds and therefore, solutions of the inclusion problem \eqref{inc1} exist and are uniquely determined. 
\end{remark}

	\subsection{Definitions of stability}
	
	To conclude this section, we recall several stability notions associated with  	the differential equation
	\begin{equation}\label{JJS1}
		\dot{z}(t)=\Phi(z(t)),\quad t\ge 0,
	\end{equation}
	where $T:\mathscr{H}\to \mathscr{H}$ is continuous. Recall that a point $z^*\in \mathscr{H}$ is an equilibrium point of \eqref{JJS1} if 
	$T(z^*)=0$. It is called \emph{Lyapunov stable} if for every $\varepsilon>0$ 
	there exists $r>0$ such that any solution with $z(0)\in B(z^*,r)$ remains in 
	$B(z^*,\varepsilon)$ for all $t>0$.

	\begin{definition}
		An equilibrium point $z^*$ is said to be
		\begin{enumerate}[(a)]
			\item \emph{finite-time stable} if it is Lyapunov stable and the corresponding settling time is finite;
			
			\item \emph{globally finite-time stable} if the above property holds for all initial conditions in $ \mathscr{H}$;
			
			\item \emph{globally fixed-time stable} if it is globally finite-time stable and the settling time is uniformly bounded for all initial conditions in $ \mathscr{H}$;
			
			\item \emph{predefined-time stable} if it is globally fixed-time stable and the system parameters can be adjusted such that the settling time can be prescribed in advance.
		\end{enumerate}
	\end{definition}

	Regarding fixed-time stability, Polyakov \cite{polyakov} proposed a Lyapunov-based 
	sufficient condition ensuring that the convergence time is uniformly bounded with 
	respect to the initial conditions. We recall this result below for completeness.
	
	\begin{theorem}\label{lm1}
		(Lyapunov condition for fixed-time stability).
		An equilibrium point $z^*$ of \eqref{JJS1} is fixed-time stable if there exists a 
		radially unbounded, continuously differentiable function 
		$\Xi: \mathscr{H}\to\mathbb{R}$ such that
		\[
		\Xi(z^*)=0,
		\qquad
		\Xi(z)>0,
		\]
		for all $z\in  \mathscr{H}\setminus\{z^*\}$, and
		\begin{equation}\label{inq of Lyapunov}
			\dot{\Xi}(z)
			\le
			-\big(M_1\Xi(z)^{k_1}+M_2\Xi(z)^{k_2}\big),
		\end{equation}
		for all $z\in  \mathscr{H}\setminus\{z^*\}$, where $M_1,M_2,k_1,k_2$ are positive  with 
		$k_1<1$ and $k_2>1$. 
		In this case, the settling time satisfies
		\[
		T(z(0))
		\le
		\frac{1}{M_1(1-k_1)}+\frac{1}{M_2(k_2-1)},
		\qquad
		\forall z(0)\in \mathscr{H}.
		\]
		In addition, by choosing 
		$k_1=1-\frac{1}{2\alpha}$ and $k_2=1+\frac{1}{2\alpha}$ with $\alpha>1$ in 
		\eqref{inq of Lyapunov}, the settling time can be explicitly bounded by
		\[
		T(z(0))\le T_{\max}
		=
		\frac{\pi\alpha}{\sqrt{M_1M_2}}.
		\]
	\end{theorem}
	
	Concerning predefined-time stability, Zheng \emph{et al.} \cite{Zheng} derived 
	sufficient Lyapunov conditions under which the convergence time can be prescribed 
	in advance. The corresponding result is summarized as follows.
	
	\begin{theorem}\label{lmpredefined}\cite{Zheng}
		Consider the differential equation \eqref{JJS1}. Assume that there exists a positive 
		definite, radially unbounded function 
		$\Xi: \mathscr{H}\to\mathbb{R}$ and a user-specified constant $T_p>0$ such that:
		\begin{itemize}
			\item[(i)] $\Xi(z(t))=0$ if and only if $\dot z(t)=0$;
			\item[(ii)] For any trajectory satisfying $\Xi(z(t))>0$, there exist constants 
			$a_1,a_2,a_3,k_1,k_2,T_p,K_p>0$, with $0<k_1<1$ and $k_2>1$, such that
			\begin{equation}\label{pt9}
				\dot{\Xi}(z(t))
				\le
				-\frac{K_p}{T_p}
				\big(
				a_1\Xi^{k_1}(z(t))
				+a_2\Xi^{k_2}(z(t))
				+a_3\Xi(z(t))
				\big).
			\end{equation}
		\end{itemize}
		Then the equilibrium point of \eqref{JJS1} is fixed-time stable, and convergence is 
		guaranteed within the predefined time $T_p$. Moreover, the constant $K_p$ can be 
		chosen as
		\begin{equation}\label{pt10}
			K_p
			=
			\frac{1}{a_3(1-k_1)}
			\ln\!\left(1+\frac{a_3}{a_1}\right)
			+
			\frac{1}{a_3(k_2-1)}
			\ln\!\left(1+\frac{a_3}{a_2}\right).
		\end{equation}
	\end{theorem}
	
	\section{ Predefined-time stability  analysis} 	\label{sec3}
In this section, we focus on the fixed-time stability analysis of solutions to the inclusion problem \eqref{inc1}. 
To this end, we introduce a novel differential equation whose equilibrium points are shown to characterize the solutions of the generalized inverse mixed variational inequality problem. 
We then establish the fixed-time convergence of the proposed differential equation.

	We first consider the following differential equation, which was studied in \cite{NHVF} and referred to as a nominal differential equation:
	
	\begin{equation} \label{hdl2}
		\dot{z}=-\mu \Big(z-J_{\gamma F}(z-\gamma G(z))\Big), \quad \mu>0.
	\end{equation}
	
	For convenience of notation, we define $$\Phi(z):=z-J_{\gamma F}(z-\gamma G(z)).$$
	
	Then the system \eqref{hdl2} becomes 
	\begin{equation} \label{hdl2n}
		\dot{z}=-\mu \Phi(z), \quad \mu>0.
	\end{equation}


Proposition \ref{rem2} identifies the equivalence between the solutions of the inclusion problem \eqref{inc1} and the equilibrium points of the system described in \eqref{hdl2}.
	\begin{proposition}\label{rem2} \cite{NHVF}
	Consider an $\eta_F$-monotone set-valued mapping $F : \mathscr{H} \to 2^\mathscr{H}$ and a single-valued operator $G : \mathscr{H} \to \mathscr{H} $, along with a positive parameter $\gamma$ such that $1 + \gamma \eta_F > 0$. Then,  \(z^*\in \mathscr{H}\) is a zero point of $F+G$  if and only if it is an equilibrium point of \eqref{hdl2n}.
	\end{proposition}


	
With the aim of investigating fixed-time stability for the inclusion problem  \eqref{inc1}, we introduce the following differential equation:

	\begin{equation}\label{newydynamicalsystem} 
		\dot{z}=-\frac{K_p}{T_p}\omega(z)\Phi(z),
	\end{equation} 
	where 
	\begin{equation}\label{dnrho} 
		\omega(z):=\begin{cases} 
			b_1\dfrac{1}{\|\Phi(z)\|^{1-p_1}}& +b_2\dfrac{1}{\|\Phi(z)\|^{1-p_2}}+b_3\dfrac{1}{\|\Phi(z)\|^{p_3}} \\
			& \mbox{ if } w\in   \mathscr{H}\setminus {\rm Zer}(\Phi),\\
			0 & \mbox{otherwise}, 
		\end{cases}
	\end{equation} 
	with  \(b_1, b_2>0, b_3> 0, p_1\in (0, 1)\), $p_3\geq 0$ and \(p_2>1\).
	\begin{remark}
		When \(b_3=0\), the system \eqref{newydynamicalsystem} is the one studied in \cite{NHVF}. In addition, if $b_2=0$, the system \eqref{newydynamicalsystem} reduces to the one examined in \cite{NH} for finite stability of a generalized monotone inclustion problem. 
		
	\end{remark}
	We have the following relationshi between an equilibrium points of \eqref{newydynamicalsystem}, \eqref{hdl2n} and the solution to the inclusion problem \eqref{inc1}.
	\begin{remark}\label{rmsol}
		A point $ z^*\in  H$ is an equilibrium point of \eqref{newydynamicalsystem}  $\iff$  it is also an equilibrium point of \eqref{hdl2n} $\iff$ it  is a zero point of $F+G$.
	\end{remark}

	\begin{proposition}\label{unique sol da sy}
		Consider the differential equation \eqref{newydynamicalsystem}. 
		Assume that the mapping $\Phi: \mathscr{H}\to \mathscr{H}$ is locally Lipschitz 
		continuous and satisfies
		\begin{equation*}
			\Phi(z^*)=0
			\quad\text{and}\quad
			\langle z-z^*,\Phi(z)\rangle>0,
			\quad
			\forall z\in \mathscr{H}\setminus\{z^*\},
		\end{equation*}
		for some $z^*\in \mathscr{H}$. 
		Then the right-hand side of \eqref{newydynamicalsystem} is continuous on 
		$ \mathscr{H}$. Moreover, for any initial condition, the differential equation 
		\eqref{newydynamicalsystem} admits a unique solution defined for all 
		$t\geq 0$.
	\end{proposition}
	
	\begin{proof}
		The result follows by applying the same line of reasoning as in the proof of 
		Lemma~3 in \cite{Zheng}, which guarantees existence and uniqueness of solutions 
		under the stated regularity and monotonicity conditions.
	\end{proof}

	We are now ready to present the main result of this section.
	
	\begin{theorem}\label{tr2}
		Consider an $\eta_F$-monotone set-valued mapping $F : \mathscr{H} \to 2^\mathscr{H}$ and a single-valued operator $G : \mathscr{H} \to \mathscr{H} $, along with a positive parameter $\gamma$ such that  Assumption {\bf (A)} is satisfied. 
		Assume that $z^*\in \mathscr{H}$ is the unique equilibrium point of the differential equation 
		\eqref{newydynamicalsystem}. 
		Then the point $z^*\in \mathscr{H}$, which solves the inclusion problem \eqref{inc1}, is a 
		fixed-time stable equilibrium of \eqref{newydynamicalsystem} for any 
		$p_1\in(0,1)$, $p_2>1$, and $p_3>0$. Moreover, the settling time admits the estimate
		\[
		T(z(0))\leq T_{\max}
		=\frac{1}{M_1(1-r_1)}+\frac{1}{M_2(r_2-1)},
		\]
		for some constants $M_1>0$, $M_2>0$, $r_1\in(0.5,1)$, and $r_2>1$.
		
		Furthermore, by choosing
		$r_1=1-\frac{1}{2\zeta}$ and $r_2=1+\frac{1}{2\zeta}$ with $\zeta>1$, the above bound 
		can be refined to
		\[
		T(z(0))\leq T_{\max}
		=\frac{\pi\zeta}{\sqrt{b_1 b_2}},
		\]
		for some constants $b_1>0$, $b_2>0$, and $\zeta>1$.
		
		In particular, when $p_3=0$, system \eqref{newydynamicalsystem} reaches the 
		equilibrium point within a prescribed time $T_p$, where $T_p$ is a user-defined 
		parameter. In this case, the control gain $K_p$ is given by
		\begin{align} \label{pretime}
			K_p=& 
			\frac{1}{b_3(1-c)(1-p_1)}
			\ln\!\left(
			1+\frac{b_3 2^{\frac{1-p_1}{2}}(1+c)^{1-p_1}}{b_1}
			\right) \notag \\ 
			&+
			\frac{1}{b_3(1-c)(p_2-1)}
			\ln\!\left(
			1+\frac{b_3 2^{\frac{1-p_2}{2}}(1-c)^{1-p_2}}{b_2}
			\right).
		\end{align} 
	\end{theorem}
	
	\begin{proof} 
		By Theorem \ref{rem2} the inclusion problem has a unique solution. By Lemma \ref{tr1}, all assumptions in Proposition \ref{unique sol da sy} are satisfied. Hence, the system \eqref{newydynamicalsystem} has a unique solution. Consider the following Lyapunov function\\
		$$\Xi(z)=\frac{1}{2}\left\|z-z^*\right\|^{2}.$$
		Then 
		\begin{align}\label{eq12} 
			\dfrac{d}{dt}\dot{\Xi}(z(t))=&\langle z-z^*, \dot{z}\rangle \notag \\
			= &- \frac{K_p}{T_p}\Big(b_1\dfrac{1}{\|\Phi(z)\|^{1-p_1}} +b_2\dfrac{1}{\|\Phi(z)\|^{1-p_2}}+b_3\dfrac{1}{\|\Phi(z)\|^{p_3}}\Big)  \langle z-z^*, \Phi(z)\rangle \notag \\
			\leq &-\frac{K_p}{T_p} \left( b_1\left(1-c \right) \dfrac{\|z-z^*\|^2}{\|\Phi(z)\|^{1-p_1}}-b_2\big(1-c\big) \dfrac{\|z-z^*\|^2}{\|\Phi(z)\|^{1-p_2}}-(1-c)b_3  \dfrac{\|z-z^*\|^2}{\|\Phi(z)\|^{p_3}}  \right)
		\end{align}
		for all $z\in   \mathscr{H}\setminus \{z^*\}$. 
		Using Part \eqref{petr1} of Lemma \ref{tr1}  it holds that 
		\begin{align}\label{nam1}
			\dot{\Xi} \leq &\frac{K_p}{T_p}\left(-\dfrac{b_1(1-c)}{(1+c)^{1-p_1}}\|z-z^*\|^{1+p_1}-(1-c)^{p_2}\Big( b_2+b_3\|\Phi(z)\|^{1-p_2-p_3}\Big) \|z-z^*\|^{1+p_2} \right)\notag\\
			\leq & -C_1(p_1)\|z-z^*\|^{1+p_1}-C_2(p_2, z) \|z-z^*\|^{1+p_2},
		\end{align}
		here \(C_1(p_1)=\frac{K_p}{T_p}\cdot \dfrac{b_1(1-c)}{(1+c)^{1-p_1}}\) and \(C_2(p_2, z)= \frac{K_p}{T_p}(1-c)^{p_2}\Big( b_2+b_3\|\Phi(z)\|^{1-p_2-p_3}\Big)\). 
		Because $1-c>0$ and $b_i>0, i=1,2,3$ it holds that $C_1(p_1)>0, C_2(p_2, z)>0$ for all $0<k_1<1, k_2>1$. By putting $$M_i =2^{\frac{1+p_1}{2}}\frac{K_p}{T_p}\cdot \dfrac{b_1(1-c)}{(1+c)^{1-p_1}}$$ and let 
		$$M_2 = 2^{\frac{1+p_2}{2}} \frac{K_p}{T_p}(1-c)^{p_2}b_2$$
		it follows from \eqref{nam1} that 
		\begin{align}
			\dot{\Xi} \leq &  -\Big(M_1\Xi(z)^{r_1}+M_2\Xi(z)^{r_2}\Big),\label{dgV} 
		\end{align}
		where  \(r_i=\frac{1+p_i}{2}\), $i=1,2$. Observe that \(M_1>0, 0.5<r_1<1\) for all \(p_1\in (0, 1)\) and \(M_2>0, r_2>1\) for all \(p_2>1\). Therefore, the desired conclusion follows from Theorem \ref{lm1}.

		When $p_3=0$, using \eqref{eq12} we get 
		\begin{align}\label{pt23}
			\dot{\Xi}&=\left\langle z-z^*, \dot{z}\right\rangle \notag \\ 
			& =-\frac{K_p}{T_p}\left\langle z-z^*, b_{1} \frac{\Phi(z)}{\|\Phi(z)\|^{1-p_{1}}}+b_{2} \frac{\Phi(z)}{\|\Phi(z)\|^{1-p_{2}}}+b_{3} \Phi(z)\right\rangle  \notag \\
			& \leq \frac{K_p}{T_p}\left(-b_{1} (1-c) \frac{\left\|z-z^*\right\|^{2}}{\|\Phi(z)\|^{1-p_{1}}}-b_{2} (1-c) \frac{\left\|z-z^*\right\|^{2}}{\|\Phi(z)\|^{1-p_{2}}}-b_{3} (1-c)\left\|z-z^*\right\|^{2}\right) \notag \\
			& \leq \frac{K_p}{T_p}\left(-\frac{b_{1} (1-c)}{(1+c)^{1-p_{1}}}\left\|z-z^*\right\|^{1+p_{1}}-b_{2} (1-c) \frac{\left\|z-z^*\right\|^{2}}{\|\Phi(z)\|^{1-p_{2}}}-b_{3} (1-c)\left\|z-z^*\right\|^{2}\right) \notag \\
			& \leq  \frac{K_p}{T_p}\left(-\frac{b_{1} (1-c)}{(1+c)^{1-p_{1}}}\left\|z-z^*\right\|^{1+p_{1}}-b_{2} (1-c)^{p_{2}}\left\|z-z^*\right\|^{1+p_{2}}-b_{3} (1-c)\left\|z-z^*\right\|^{2}\right) \notag \\
			& = \frac{K_p}{T_p}\left(-\frac{b_{1} (1-c)}{(1+c)^{1-p_{1}}}(2 V)^{\frac{1+p_{1}}{2}}-b_{2} (1-c)^{p_{2}}(2 V)^{\frac{1+p_{2}}{2}}-2 b_{3} (1-c) \Xi\right) \notag\\
			& =\frac{K_p}{T_p}\left(-\left(\sqrt{2}^{1+p_{1}} \frac{b_{1} (1-c)}{(1+c)^{1-p_{1}}} \Xi^{\frac{1+p_{1}}{2}}+\sqrt{2}^{1+p_{2}} b_{2} (1-c)^{p_{2}} \Xi^{\frac{1+p_{2}}{2}}+2 b_{3} (1-c)\Xi\right)\right) \notag \\
			& =\frac{K_p}{T_p}\left(-a_{1} \Xi^{\frac{1+p_{1}}{2}}-a_{2} \Xi^{\frac{1+p_{2}}{2}}-a_{3} \Xi\right) 
		\end{align}
		where $a_{1}=\sqrt{2}^{1+p_{1}} \frac{b_{1} (1-c)}{(1+c)^{1-p_{1}}}>0, a_{2}=\sqrt{2}^{1+p_{2}} b_{2} (1-c)^{p_{2}}>0, a_{3}=2 b_{3} (1-c)$.  Let
		\begin{align*}
			K_p= & \frac{2}{a_{3}\left(1-p_{1}\right)} \ln \left(1+\frac{a_{3}}{a_{1}}\right)+\frac{2}{a_{3}\left(p_{2}-1\right)} \ln \left(1+\frac{a_{3}}{a_{2}}\right) \\
			= & \frac{1}{b_{3} (1-c)\left(1-p_{1}\right)} \ln \left(1+\frac{b_{3} \sqrt{2}^{1-p_{1}} (1+c)^{1-p_{1}}}{b_{1}}\right) \\
			& +\frac{1}{b_{3} (1-c)\left(p_{2}-1\right)} \ln \left(1+\frac{b_{3} \sqrt{2}^{1-p_{2}} (1-c)^{1-p_{2}}}{b_{2}}\right).
		\end{align*}
		Then, due to \eqref{pt23} and Lemma~\ref{lmpredefined},  the trajectary $z(t)$ of the differential equation \eqref{newydynamicalsystem} reaches its  equilibrium point   within the predefined time \(T_p\). In addition, for any initial condition \(z(0)\in  \mathscr{H}\), the settling time satisfies
		\[
		T(z(0)) \le T_{\max} = T_p .
		\]
		Thus, we complete the proof. 
	\end{proof}

	\section{ Discretization of the differential equation} \label{sec4}
	This section addresses the construction of a discrete-time scheme associated 
	with the considered inclusion problem. By means of a forward Euler 
	approximation in time, we derive an iterative algorithm and verify that the 
	fixed-time stability property of the continuous system is retained after 
	discretization. For this purpose, we begin with a general framework.
	
	Specifically, we consider the differential inclusion
	\begin{equation}\label{eq20}
		\dot{z} \in \Psi(z),
	\end{equation}
	and introduce the following standing assumption.
	
	\noindent ({\bf B})
	Let $\Psi: \mathscr{H} \to 2^{ \mathscr{H}}$ be an upper semi-continuous set-valued mapping with 
	nonempty, convex, and compact values. Assume that there exists 
	$z^*\in \mathscr{H}$ such that $0\in\Psi(z^*)$. Moreover, suppose that there is a 
	positive definite, radially unbounded, locally Lipschitz continuous, and 
	regular function $\Xi: \mathscr{H}\to\R$ satisfying $\Xi(z^*)=0$ and
	\begin{equation*}
		\sup \dot{\Xi}(z)
		\leq
		-\left(
		M_1 \Xi(z)^{1-\frac{1}{\alpha}}
		+M_2 \Xi(z)^{1+\frac{1}{\alpha}}
		\right)
	\end{equation*}
	for all $z\in \mathscr{H}\setminus\{z^*\}$, where $M_1,M_2>0$ and $\alpha>1$. 
	Here,
	\begin{equation*}
		\dot{\Xi}(z)
		:=
		\left\{
		x\in\R:
		\exists w\in\Psi(z):		\langle v,w\rangle=x,
		\ \forall v\in\partial_c\Xi(z)
		\right\},
	\end{equation*}
	with $\partial_c\Xi(z)$ denoting Clarke’s generalized gradient of $\Xi$ at $z$.
	
	\medskip
	\noindent
	We next discretize \eqref{eq20} using a forward Euler scheme with step size 
	$h_n>0$, which leads to
	\begin{equation*}
		\frac{z(t_{n+1})-z(t_n)}{h_n}
		\in
		\Psi\big(z(t_n)\big).
	\end{equation*}
	By fixing $h_n=\beta$ and writing $z(t_n)=z_n$ for $n\in\N$, we obtain the 
	recursive inclusion
	\begin{equation}\label{eq24}
		z_{n+1}\in z_n+\beta\,\Psi(z_n).
	\end{equation}
	
	The following result, adapted from \cite{Garg21}, provides an explicit estimate 
	on the deviation of the sequence generated by \eqref{eq24}.
	
	\begin{theorem}\cite{Garg21, POL}\label{tr4}
		Let $z^*$ be an equilibrium point of \eqref{eq20}. 
		Suppose that Assumption {\bf (B)} holds and that $\Xi$ satisfies the quadratic 
		growth condition
		\begin{equation*}
			\Xi(z)\geq c\|z-z^*\|^2
		\end{equation*}
		for all $z\in \mathscr{H}$ and some $c>0$. 
		Then, for any initial point $z_0\in \mathscr{H}$ and any $\varepsilon>0$, there exists 
		$\beta^*>0$ such that, for all $\beta\in(0,\beta^*]$, the iterates generated 
		by \eqref{eq24} satisfy
		\begin{equation*}
			\|z_n-z^*\|<
			\begin{cases}
				\dfrac{1}{\sqrt{c}}
				\left(
				\sqrt{\dfrac{M_1}{M_2}}
				\tan\!\left(
				\dfrac{\pi}{2}
				-\dfrac{\sqrt{M_1M_2}}{\alpha}\beta n
				\right)
				\right)^{\frac{\alpha}{2}}
				+\varepsilon,
				& n\leq n^*,\\[2mm]
				\varepsilon,
				& n> n^*,
			\end{cases}
		\end{equation*}
		where
		\[
		n^*=
		\left\lceil
		\dfrac{\alpha\pi}{2\gamma\sqrt{M_1M_2}}
		\right\rceil.
		\]
	\end{theorem}
	
	We now apply the above estimate to the discretized form of 
	\eqref{newydynamicalsystem}.
	
	\begin{theorem}\label{dl cuoi}
		Consider the forward discretization of \eqref{newydynamicalsystem} given by
		\begin{equation}\label{eq29}
			z_{n+1}
			=
			z_n
			-
			\beta \frac{K_p}{T_p}\,
			\omega(z_n)\Phi(z_n),
		\end{equation}
		where $\omega$ is defined in \eqref{dnrho}, $b_1,b_2>0$, $b_3>0$, 
		$p_1=1-\frac{2}{\alpha}$, $p_2=1+\frac{2}{\alpha}$, $\alpha\in(2,\infty)$, and 
		$\beta>0$ denotes the step size. 
		Assume that Assumption {\bf (A)} holds. 
		Then, for any $z_0\in \mathscr{H}$ and any $\varepsilon>0$, there exist constants 
		$\alpha>2$, $M_1,M_2>0$, and $\beta^*>0$ such that, for all 
		$\beta\in(0,\beta^*]$,
		\begin{equation*}
			\|z_n-z^*\|<
			\begin{cases}
				\sqrt{2}
				\left(
				\sqrt{\dfrac{M_1}{M_2}}
				\tan\!\left(
				\dfrac{\pi}{2}
				-\dfrac{\sqrt{M_1M_2}}{\alpha}\gamma n
				\right)
				\right)^{\frac{\alpha}{2}}
				+\varepsilon,
				& n\leq n^*,\\[2mm]
				\varepsilon,
				& n> n^*,
			\end{cases}
		\end{equation*}
		where
		\[
		n^*=
		\left\lceil
		\dfrac{\alpha\pi}{2\gamma\sqrt{M_1M_2}}
		\right\rceil,
		\]
		$z_n$ is defined by \eqref{eq29} with initial point $z_0$, and $z^*\in \mathscr{H}$ is 
		the unique solution of the inclusion problem \eqref{inc1}.
	\end{theorem}
	
	\begin{proof}
		According to the proof of Theorem~\ref{tr2}, inequality \eqref{dgV} holds for any 
		$r_1\in(0,1)$ and $r_2>1$. 
		Since $r_1=1-\frac{2}{\alpha}$ and $r_2=1+\frac{2}{\alpha}$, these conditions are 
		fulfilled for all $\alpha>2$. 
		Consequently, all assumptions of Theorem~\ref{tr4} are satisfied by selecting 
		the Lyapunov function $\Xi(z)=\frac{1}{2}\|z-z^*\|^2$, for which $c=\frac{1}{2}$. 
		The conclusion then follows directly from Theorem~\ref{tr4}.
	\end{proof}

	\section{Applications}\label{sec5}
	In this section, we demonstrate how the proposed prdefined-time stable differential equation \eqref{newydynamicalsystem} can be employed to solve several classes of optimization-related problems. 
	In particular, we focus on constrained optimization problems (COPs), mixed variational inequality problems (MVIPs), and classical variational inequality problems (VIPs).
	
	\subsection{Application to constrained optimization problems}
	
	We first consider the constrained optimization problem
	\begin{equation}\label{cop1}
		\min_{z\in \mathscr{H}}\; h(z)+\phi(z),
	\end{equation}
	where $h:\mathscr{H}\to\R$ is a continuously differentiable and convex function, while 
	$\phi:\mathscr{H}\to\R$ is proper, convex, and lower semicontinuous. 
	Note that $\phi$ is not necessarily differentiable, and problem \eqref{cop1} reduces to an unconstrained optimization problem when $\phi\equiv 0$.
	
	Define the operator $A:=\partial \phi+\nabla h$. 
	Then solving \eqref{cop1} is equivalent to finding $z^*\in \mathscr{H}$ such that
	\[
	0\in A(z^*),
	\]
	which coincides with the inclusion problem \eqref{inc1} by taking $F=\partial\phi$ and $G=\nabla h$.
	In this setting, the resolvent operator satisfies
	\[
	J_{\gamma F}(z-\gamma G(z))=\prox_{\gamma\phi}(z-\gamma\nabla h(z)).
	\]
	Consequently, the operator $\Phi$ appearing in \eqref{newydynamicalsystem} takes the form
	\begin{equation}\label{dDNMCOP}
		\Phi(z)=z-\prox_{\gamma\phi}(z-\gamma\nabla h(z)).
	\end{equation}
	
	Following \cite{JU5}, we impose the following condition.
	
	\begin{itemize}
		\item[{\bf (A1)}] The function $ \mathscr{H}$ is strongly monotone with constant $\eta>0$, its gradient $\nabla h$ is Lipschitz continuous with constant $L>0$, and the stepsize $\gamma$ satisfies $\gamma L^2<2\eta$.
	\end{itemize}
	
	Since strong monotonicity of $ \mathscr{H}$ implies strong monotonicity of $\nabla h$ with the same modulus, assumptions {\bf(A)} and {\bf(A')} are fulfilled whenever {\bf(A1)} holds.
	Therefore, Theorem~\ref{tr2} can be directly applied to obtain fixed-time convergence for problem \eqref{cop1}.
	
	\begin{proposition} Consider the constrained optimization problem \eqref{cop1}. 	Assume that condition {\bf(A1)} holds. 
		Let $z^*$ be a solution of \eqref{cop1}.  
		Then $z^*$ is a fixed-time stable equilibrium point of the differential equation \eqref{newydynamicalsystem} with $\Phi$ defined by \eqref{dDNMCOP}. Moreover, the settling time satisfies
		\[
		T(z(0))\leq T_{\max}
		=\frac{1}{M_1(1-r_1)}+\frac{1}{M_2(r_2-1)},
		\]
		where $M_1>0$, $M_2>0$, $r_1\in(0.5,1)$, and $r_2>1$.
		
		In addition, if $r_1=1-\frac{1}{2\alpha}$ and $r_2=1+\frac{1}{2\alpha}$ for some $\alpha>1$, then
		\[
		T(z(0))\leq T_{\max}=\frac{\pi\alpha}{\sqrt{M_1M_2}}.
		\]
		
		Furthermore, when $p_3=0$, the equilibrium point of \eqref{newydynamicalsystem} is predefined-time stable, with the constant $K_p$ given in \eqref{pretime}.
	\end{proposition}
	
	\subsection{Application to mixed variational inequality problems}
	
	We next consider the mixed variational inequality problem of finding 
	\begin{equation}\label{mvip1}
		z^*\in \mathscr{H} \text{ such that }
		\langle G(z^*),z-z^*\rangle+\phi(z^*)-\phi(z)\geq0,
		\quad \forall z\in \mathscr{H},
	\end{equation}
	where $G:\mathscr{H}\to \mathscr{H}$ is an operator and $\phi:\mathscr{H}\to\R$ is proper, convex, and l.s.c.
	Problem \eqref{mvip1} can be equivalently reformulated as the inclusion problem \eqref{inc1} with $F=\partial\phi$.
	
	In this case, the operator $\Phi$ in \eqref{newydynamicalsystem} becomes
	\begin{equation}\label{PhiMVIP}
		\Phi(z)=z-J_{\gamma F}(z-\gamma G(z))
		=z-\prox_{\gamma\phi}(z-\gamma G(z)).
	\end{equation}
	
	Note that with $p_3=0$, this differential equation was previously investigated in \cite{Zheng} for predefined-time convergence of MVIPs.
	In contrast, the case $p_3=1$  was  studied \cite{JU2022} , where fixed-time stability of MVIPs was established. Assume that $\phi$ is convex, $G$ is $\eta$-strongly monotone and $L$-Lipschitz continuous, and that $\gamma\in(0,2\eta/L^2)$.
	Then, by Theorem~\ref{tr2}, the solution $z^*$ of \eqref{mvip1} is a fixed-time stable equilibrium point, and the settling time satisfies
	\[
	T(z(0))\leq T_{\max}
	=\frac{1}{M_1(1-r_1)}+\frac{1}{M_2(r_2-1)},
	\]
	for some $M_1,M_2>0$, $r_1\in(0.5,1)$, and $r_2>1$.
	Moreover, choosing $r_1=1-\frac{1}{2\alpha}$ and $r_2=1+\frac{1}{2\alpha}$ with $\alpha>1$ yields
	\[
	T(z(0))\leq T_{\max}=\frac{\pi\alpha}{\sqrt{M_1M_2}}.
	\]
	When $p_3=0$, predefined-time stability is also guaranteed.
	
	\subsection{Application to variational inequality problems}
	
	Finally, we consider the variational inequality problem:
	\begin{equation}\label{vip1}
		\text{find } z^*\in\Omega \text{ such that }
		\langle G(z^*),z-z^*\rangle\geq0,
		\quad \forall z\in\Omega,
	\end{equation}
	where $\Omega\subset H$ is nonempty, closed, and convex, and $G:\Omega\to \mathscr{H}$.
	Problem \eqref{vip1} is a special case of \eqref{mvip1} with $\phi\equiv0$.
	
	The variational inequality \eqref{vip1} can be written as the inclusion $0\in (N_\Omega+G)(z)$.
	In this case, the resolvent reduces to the metric projection,
	\[
	J_{\gamma N_\Omega}(z-\gamma G(z))=P_\Omega(z-\gamma G(z)),
	\]
	and the operator $\Phi$ becomes
	\begin{equation}\label{dnm3}
		\Phi(z)=z-P_\Omega(z-\gamma G(z)).
	\end{equation}
	
If  $G$ is strongly monotone with modulus $\eta>0$ and Lipschitz continuous with constant $L>0$, then	Theorem~\ref{tr2} ensures fixed-time convergence of \eqref{dnm3}. Also when $b_3=0$ the differential equation \eqref{newydynamicalsystem} with \(\Phi\) defined by \eqref{dnm3} is the one studied in  \cite{Garg21}.
	Again, if $p_3=0$, the trajectory generated by the differential equation
	\eqref{newydynamicalsystem} reaches a solution of the inclusion problem
	\eqref{inc1} within a predefined time. In this case, the predefined-time
	constant $K_p$ is given by \eqref{pretime}.

\section{Numerical Illustration}\label{num}
This section is devoted to numerical experiments illustrating the theoretical findings obtained in the preceding sections.

The continuous-time dynamical systems under consideration are solved numerically by means of the $scipy.integrate.solve\_ivp$ function
 from the SciPy library, implemented on Google Colab, over the interval $[0, T_{\max}]$. The following example is adopted from \cite{NHVF}.

\begin{example}\label{ex1}
We consider the monotone inclusion problem \eqref{inc1} with
\[
F = I, \qquad G = M^{\top}M,
\]
where $I$ denotes the identity matrix of size $n$ and $M$ is an arbitrary matrix in
$\mathbb{R}^{m\times n}$. It is straightforward to verify that the inclusion
\[
0 \in F(z) + G(z)
\]
admits a unique solution given by $z^* = 0$.

For the numerical implementation, we set $m=10$, $n=8$, $b_1=20$, $b_2=200$,
$p_1=0.99$, $p_2=1.01$, $p_3=0$, and $\beta = 0.005$.
The dynamical system \eqref{newydynamicalsystem} is then applied with
$T_{\max}=5$.
Figure~\ref{fig1} illustrates the trajectory of the generated solution, which
converges to the equilibrium point $z^*$, together with the corresponding error
evolution $\|z(t)-z^*\|^2$.

\begin{figure}
	\begin{centering}
		\includegraphics[width=\textwidth]{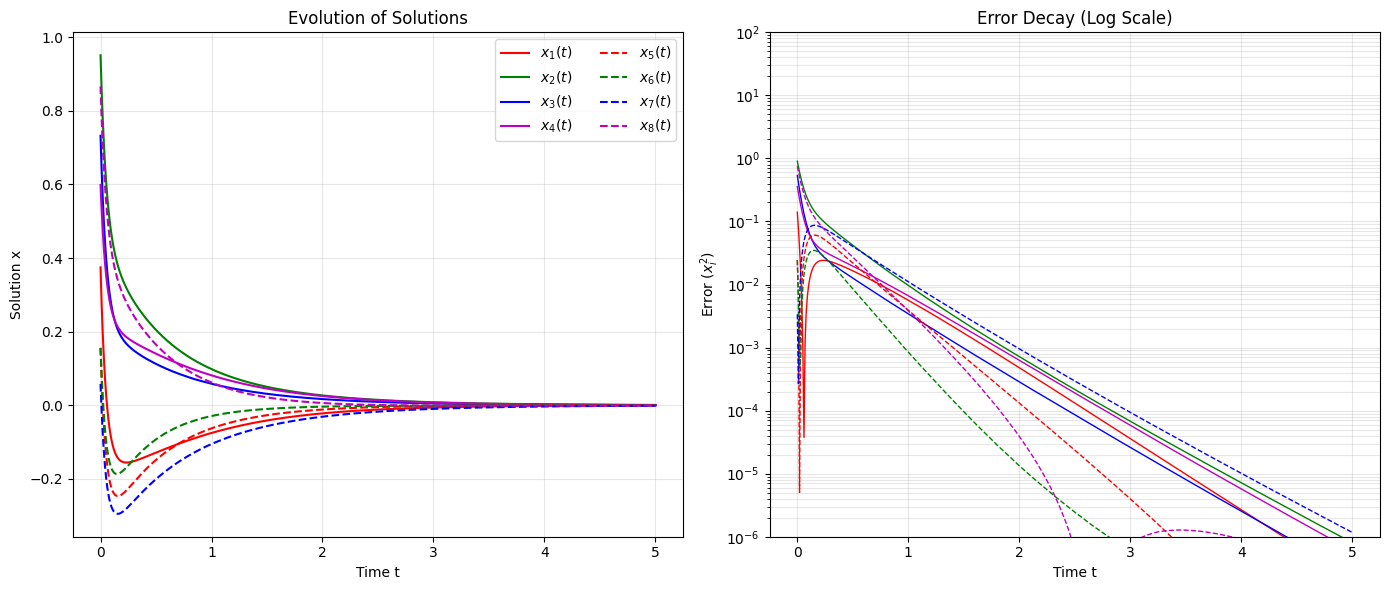}
		\par\end{centering}
	\protect\caption{Transient responses and error decay of the dynamical system \eqref{newydynamicalsystem} for Example~\ref{ex1}.
		\label{fig1}}
\end{figure}
To assess the convergence rate of the proposed method in comparison with the one
introduced in \cite{NHVF}, we perform numerical experiments with different choices
of the parameters $b_3$ and $p_3$. It is worth noting that, when $b_3 = 0$, the
dynamical system \eqref{newydynamicalsystem} reduces to the model studied in
\cite{NHVF}. The numerical results indicate that introducing the additional term
associated with $p_3 \neq 0$ leads to a faster convergence rate.

Figure~\ref{fig1nam} depicts the evolution of the error for $m=10$, $n=8$,
$b_1=20$, $b_2=200$, $p_1=0.99$, $p_2=1.01$, and $\beta = 0.005$, under different
choices of $(b_3,p_3)$, namely $(b_3,p_3)=(0,0)$, $(5,0)$, and $(5,1)$.
The dynamical system \eqref{newydynamicalsystem} is simulated over the time
interval $[0,T_{\max}]$ with $T_{\max}=5$. The results clearly show that the error
$\|z(t)-z^*\|^2$ decays more rapidly when $p_3\neq 0$, confirming the improved
convergence behavior of the proposed scheme.

\begin{figure}[H]
	\begin{centering}
		\includegraphics[width=\textwidth]{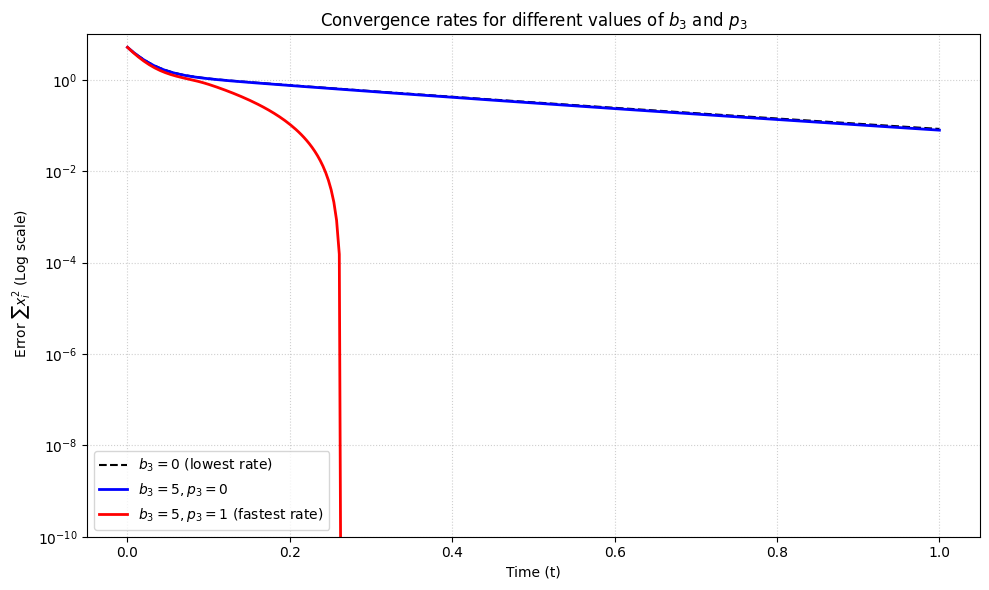}
		\par\end{centering}
	\protect\caption{Convergence
		rate with different values of parameters of dynamic system \eqref{newydynamicalsystem} for Example \ref{ex1} .
		\label{fig1nam}}
\end{figure}

To investigate the influence of $p_3$ on the convergence rate, Figure~\ref{fig3}
illustrates the convergence behavior for different values of $p_3$, while all
other parameters are kept fixed.

From Figure~\ref{fig3}, we observe that the largest value \(p_3 = 1.2\) yields the fastest convergence rate compared with \(p_3 = 1\) and \(p_3 = 0.5\). Moreover, the convergence rate for \(p_3 = 1\) is higher than that obtained for \(p_3 = 0.5\).

\begin{figure}
	\begin{centering}[H]
		\includegraphics[width=\textwidth]{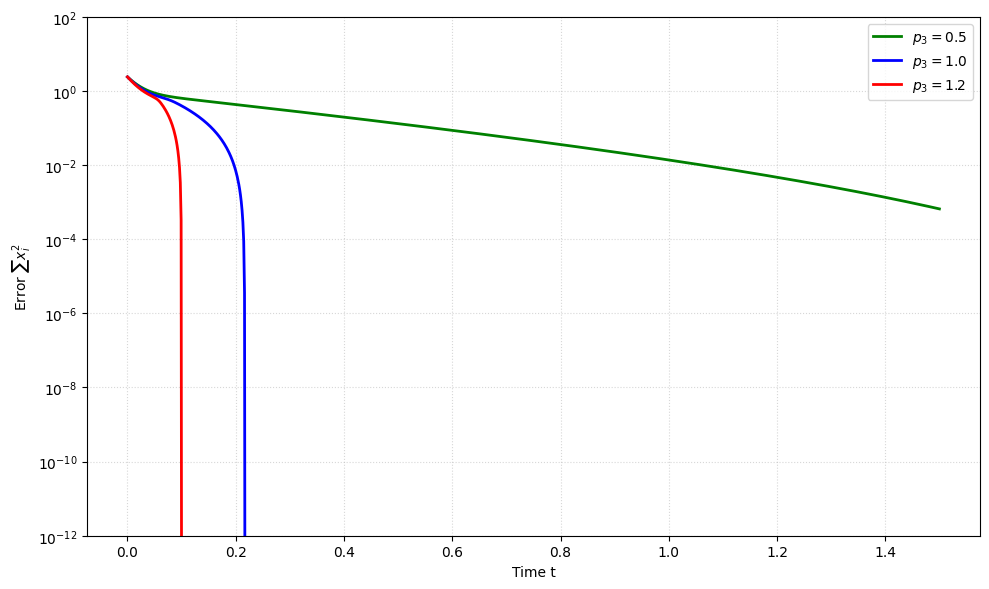}
		\par\end{centering}
	\protect\caption{Convergence
		rate with different values of $p_3$ of dynamic system \eqref{newydynamicalsystem} for Example \ref{ex1}.
		\label{fig3}}
\end{figure}

\newpage

\section{Conclusion}\label{sec6}

This paper proposes a modified forward–backward splitting dynamical system with predefined-time stability guarantees. Under a generalized monotonicity framework for the involved operators, we establish the existence and uniqueness of system trajectories and prove that their convergence to the unique solution of the associated inclusion problem is achieved within a user-prescribed convergence time, independent of initial conditions. Furthermore, we show that the forward Euler scheme provides a consistent time discretization of the proposed continuous-time dynamics. Several applications are also discussed, illustrating how mixed variational inequalities, variational inequalities, and convex optimization problems can be naturally embedded into the proposed inclusion framework.

Future research will focus on extending the predefined-time stability analysis of the modified forward–backward splitting dynamics to infinite-dimensional Hilbert and Banach spaces, as well as on relaxing the underlying assumptions by adopting weaker conditions similar to those considered in related works.

\end{example}

	\section*{Declarations} 
	{\bf Conflict of interest} The authors declare no competing interests.	
	
	\noindent {\bf Data availability statement:} Data sharing is not applicable to this article as no datasets were generated or analyzed during the current study. 
	
	\noindent {\bf Financial interests: } The authors thank Ho Chi Minh City University of Technology and Engineering for supporting this research.

\end{document}